\documentclass{svproc}
\usepackage{latexsym}
\usepackage{physics}
\usepackage{amssymb,amsmath,amscd}
\usepackage{mathtools}
\usepackage{graphicx}
\usepackage{hyperref}

\newcommand{\Dgm}{\mathrm{Dgm}}
\begin{document}
\mainmatter              
\title{Some new methods to build group equivariant non-expansive operators in TDA}
\titlerunning{On the construction of new classes of GENEOs}  
%


\author{Nicola Quercioli 
}

\authorrunning{Nicola Quercioli} 

\institute{
Department of Mathematics, University of Bologna, Italy
\email{nicola.quercioli2@unibo.it}
}
\maketitle              

\begin{abstract}
Group equivariant operators are playing a more and more relevant role in machine learning and topological data analysis. In this paper we present some new results concerning the construction of $G$-equivariant non-expansive operators (GENEOs) from a space $\varPhi$ of real-valued bounded continuous functions on a topological space $X$ to $\varPhi$ itself. The space $\varPhi$ represents our set of data, while $G$ is a subgroup of the group of all self-homeomorphisms of $X$, representing the invariance we are interested in.
\end{abstract}


\section*{Introduction}
In the recent years Topological Data Analysis (TDA) has imposed itself as a useful tool in order to manage huge amount of data of the present digital world~\cite{Ca09}. In particular, persistent homology has assumed a relevant role as an efficient tool for qualitative and topological comparison of data~\cite{EdMo13}, since in several applications we can express the acts of measurement by $\mathbb{R}^m$-valued functions defined on a topological space, so inducing filtrations on such a space~\cite{BiDFFa08}. These filtrations can be analyzed by means of the standard methods used in persistent homology. For further and detailed information about persistent homology we refer the reader to~\cite{EdHa08}.

The importance of group equivariance in machine learning is well known (see, e.g.,~\cite{AnRoPo16,CoWe16,MaVoKo17,MaBoBr15}). Our work on group equivariant non-expansive operators (GENEOs) is devoted to possibly establish a link between persistence theory and machine learning. Our basic idea is that acts of measurement are directly influenced by the observer, and we should mostly focus on well approximating the observer, rather than precisely describing the data (see, e.g.,~\cite{Fr16}). In some sense, we could see the observer as a collection of GENEOs acting on a suitable space of data and encode in the choice of these operators the invariance we are interested in.

The concept of invariance group leads us to consider the \emph{natural pseudo-distance} as our main tool to compare data. Let us consider two real-valued functions $\varphi$, $\psi$ on a topological space $X$, representing the data we want to compare, and a group $G$ of self-homeomorphisms of $X$. Roughly speaking, the computation of the natural pseudo-distance $d_G$ between $\varphi$ and $\psi$ is the attempt of finding the best correspondence between these two functions with respect to the invariance group $G$.

Unfortunately, $d_G$ is difficult to compute, but  \cite{FrJa16} illustrates a possible path to approximate the natural pseudo-distance by means of a dual approach involving persistent homology and GENEOs. In particular, one can see that a good approximation of the space $\mathcal{F}(\varPhi,G)$ of all GENEOs corresponds to a good approximation of the pseudo-distance $d_G$. In order to extend our knowledge about  $\mathcal{F}(\varPhi,G)$, we devote this paper to introduce some new methods to construct new GENEOs from a given set of GENEOs.

The outline of our paper follows. In Section \ref{model} we briefly present our mathematical framework. In Section \ref{powermeans} we give a new result about building GENEOs by power means and show some examples to explain why this method is useful and meaningful. In Section \ref{series} we illustrate a new procedure to build new GENEOs by means of series of GENEOs. In particular, this is a first example of costruction of an operator starting from an infinite set of GENEOs.

\section{Our mathematical model}
\label{model}

In this section the mathematical model illustrated in~\cite{FrJa16} will be briefly recalled.
Let $X$ be a (non-empty) topological space, and $\varPhi$ be a topological subspace of the topological space $C^{0}_{b}(X,\mathbb{R})$ of the continuous bounded functions from $X$ to $\mathbb{R}$, endowed with the topology induced by the sup-norm $\| \cdot \|_{\infty}$.
The  elements of $\varPhi$ represent our data and are called \textit{admissible filtering functions} on the space $X$. We also assume that $\varPhi$ contains at least the constant functions $c$ such that $ |c| \le \sup_{\varphi \in \varPhi} \| \varphi \|_\infty$.
The invariance of the space $\varPhi$ is represented by the action of a subgroup $G$ of the group $\mathrm{Homeo}(X)$ of all homeomorphisms from $X$ to itself.
The group $G$ is used to act on $\Phi$ by composition on the right, i.e. we suppose that $\varphi \circ g$ is still an element of $\varPhi$ for any $\varphi\in \Phi$ and any $g\in G$. In other words, the functions $\varphi$ and $\varphi\circ g$, elements of $\varPhi$, are considered equivalent to each other for every $g\in G$.

In this theoretical framework we use the \emph{natural pseudo-distance $d_G$} to compare functions.

\begin{definition}\label{defdG}
For every $\varphi_1,\varphi_2\in\Phi$ we can define the function $d_G(\varphi_1,\varphi_2):=\inf_{g \in G}\sup_{x \in X}\left|\varphi_1(x)-\varphi_2(g(x))\right|$ from $\varPhi \times \varPhi$ to $\mathbb{R}$. The function $d_G$ is called the \emph{natural pseudo-distance} associated with the group $G$ acting on $\Phi$.
\end{definition}

We can consider this (extended) pseudo-metric as the ground truth for the comparison of functions in $\Phi$ with respect to the action of the group $G$. Unfortunately, $d_G$ is usually difficult to compute. However, the natural pseudo-distance can be studied and approximated by a method involving \emph{$G$-equivariant non-expansive operators}.

\begin{definition}

A $G$-equivariant non-expansive operator (GENEO) for the pair $(\Phi,G)$ is a function
$$F : \Phi \longrightarrow \Phi $$
that satisfies the following properties:
\begin{enumerate}
\item F is $G$-equivariant: $F(\varphi\circ g)=F(\varphi)\circ g, \quad \forall \  \varphi\in \Phi, \quad \forall \  g \in G$;
\item F is non-expansive: $\| F(\varphi_{1})-F(\varphi_{2})\|_{\infty} \le \| \varphi_{1} -\varphi_{2}\|_{\infty}, \quad \forall \ \varphi_{1},\varphi_{2}\in \Phi$.
\end{enumerate}
\end{definition}

The symbol $\mathcal{F}(\Phi,G)$ is used to denote the set of all $G$-equivariant non-expansive operators for $(\Phi,G)$.
Obviously $\mathcal{F}(\Phi,G)$ is not empty because it contains at least the identity operator.

\begin{remark}
The non-expansivity property means that the operators in $\mathcal{F}(\Phi,G)$ are $1$-Lipschitz functions and therefore they are continuous. We underline that GENEOs are not required to be linear.
\end{remark}

If $X$ has nontrivial homology in degree $k$, the following key result holds~\cite{FrJa16}.


 \begin{theorem}\label{maintheoremforG}
$d_G(\varphi_1,\varphi_2)=\sup_{F\in \mathcal{F}(\varPhi,G)} d_{match}(\Dgm_k(F(\varphi_1)),\Dgm_k(F(\varphi_2)))$, where $\Dgm_k(\varphi)$ denotes the $k$-th persistence diagram of the function $\varphi:X\to\mathbb{R}$ and $d_{match}$ is the classical matching distance.
\end{theorem}

Persistent homology and the natural pseudo-distance are related to each other by Theorem~\ref{maintheoremforG} via GENEOs.
This result enables us to approximate $d_G$ by means of $G$-equivariant non-expansive operators. The construction of new classes of GENEOs is consequently a relevant step in the approximation of the space $\mathcal{F}(\Phi,G)$, and hence in the computation of the natural pseudo-distance, so justifying the interest for the results shown in Sections~\ref{powermeans} and \ref{series}.

\section{Building new GENEOs by means of power means}
\label{powermeans}
In this section we introduce a new method to build GENEOs, concerning the concept of power mean.
Now we recall a proposition that enables us to find new GENEOs, based on the use of $1$-Lipschitz functions (see~\cite{FrQu16}).
\begin{proposition}\label{lipfun}
Let $L$ be a 1-Lipschitz function from $\mathbb{R}^n$ to $\mathbb{R}$, where $\mathbb{R}^n$ is endowed with the norm $\|(x_1,\dots,x_n)\|_\infty=\max\{|x_1|,\dots,|x_n|\}$. Assume also that $F_1, \dots, F_n$ are GENEOs for $(\varPhi,G)$. Let us define the function $L^*(F_1, \dots , F_n):\varPhi \longrightarrow C^0_b(X,\mathbb{R})$ by setting
\begin{equation*}
    L^*(F_1, \dots , F_n)(\varphi)(x):= L(F_1(\varphi)(x), \dots , F_n(\varphi)(x)).
\end{equation*}
If $L^*(F_1, \dots , F_n)(\varPhi)\subseteq \varPhi$, the operator $L^*(F_1, \dots , F_n)$ is a GENEO for $(\varPhi,G)$.
\end{proposition}
In order to apply this proposition, we recall some definitions and properties about power means and $p$-norms.
Let us consider a sample of real numbers $x_1, \dots, x_n$ and a real number $p>0$. As well known, the power mean $M_p(x_1, \dots, x_n)$ of $x_1, \dots, x_n$ is defined by setting
\begin{equation*}
    M_p(x_1, \dots, x_n):= \left(\frac{1}{n}\sum_{i=1}^n |x_i|^p\right)^\frac{1}{p}.
\end{equation*}

In order to proceed, we consider the function $\|\cdot \|_p: \mathbb{R}^n \longrightarrow \mathbb{R}$ defined by setting
\begin{equation*}
    \| x \|_{p}=(|x_1|^p + |x_2|^p + \dots + |x_n|^p)^\frac{1}{p}
\end{equation*}
where $x=(x_1,\dots,x_n)$ is a point of $\mathbb{R}^n$. It is well know that, for $p\ge 1$, $\|\cdot\|_p$ is a norm and that for any $x \in \mathbb{R}^n$, we have $\lim_{p \to \infty}\|x\|_p=\|x\|_\infty$.
Finally, it is easy to check that if $x \in \mathbb{R}^n$ and $0< p < q < \infty$, it holds that
\begin{equation}
\label{ineqpq}
    \|x\|_q \le \|x\|_{p} \le n^{\frac{1}{p} - \frac{1}{q}} \|x\|_q.
\end{equation}

For $q$ tending to infinity, we obtain a similar inequality:
\begin{equation}
\label{ineqpinfty}
    \|x\|_\infty \le \|x\|_p \le n^\frac{1}{p} \|x\|_\infty.
\end{equation}

Now we can define a new class of GENEOs. Let us consider $F_1, \dots , F_n$ GENEOs for $(\varPhi,G)$ and $p>0$. Let us define the operator $M_p(F_1, \dots, F_n): \varPhi \longrightarrow C^0_b(X, \mathbb{R})$ by setting
\begin{equation*}
    M_p(F_1, \dots, F_n)(\varphi)(x):= M_p(F_1(\varphi)(x), \dots, F_n(\varphi)(x)).
\end{equation*}

\begin{theorem}
If $p \ge 1$ and $M_p(F_1, \dots, F_n)(\varPhi) \subseteq \varPhi$, $M_p(F_1, \dots, F_n)$ is a GENEO for $(\varPhi,G)$.
\end{theorem}

\begin{proof}
If we show that $M_p$ is a $1$-Lipschitz function for $p \ge 1$, Proposition \ref{lipfun} will ensure us that $M_p(F_1, \dots, F_n)$ is a GENEO.

Let $p \ge 1$ and $x,y \in \mathbb{R}^n$. Since $\| \cdot \|_p$ is a norm, the reverse triangle inequality holds. Therefore, because of (\ref{ineqpinfty}) we have that:
\begin{align*}
    \left| \left(\frac{1}{n}\sum_{i=1}^n |x_i|^p\right)^\frac{1}{p} - \left(\frac{1}{n}\sum_{i=1}^n |y_i|^p\right)^\frac{1}{p}\right| & = \left(\frac{1}{n}\right)^\frac{1}{p} \left| \left(\sum_{i=1}^n |x_i|^p\right)^\frac{1}{p} - \left(\sum_{i=1}^n |y_i|^p\right)^\frac{1}{p}\right|\\
    & = \left(\frac{1}{n}\right)^\frac{1}{p} \left| \| x \|_p - \| y \|_p \right|\\
    & \le \left(\frac{1}{n}\right)^\frac{1}{p} \| x - y \|_p \\
    & \le \left(\frac{1}{n}\right)^\frac{1}{p} n^\frac{1}{p} \| x - y \|_\infty = \| x - y \|_\infty.
\end{align*}
Hence, for $p \ge 1$ $M_p$ is non-expansive (i.e. $1$-Lipschitz) and the statement of our theorem is proved.

\end{proof}

\begin{remark}
If $ 0 < p<1$ and $n > 1$, $M_p$ is not a $1$-Lipschitz function. This can be easily proved by showing that for $x_2=x_3=\dots=x_n=1$ the derivative $\frac{\partial M_p}{\partial x_1}$ is not bounded.
\end{remark}

\subsection{Examples}
In this subsection we want to justify the use of the operator $M_p$. In order to make this point clear, let us consider the space $\varPhi$ of all $1$-Lipschitz functions from the unit circle $S^1$ to $[0,1]$ and the invariance group $G$ of all rotations of $S^1$. Now, we can take into consideration the following operators:
\begin{itemize}
    \item the identity operator $F_1:\varPhi \longrightarrow \varPhi$;
    \item the operator $F_2:\varPhi \longrightarrow \varPhi$ defined by setting $F_2(\varphi):= \varphi \circ \rho_\frac{\pi}{2}$ for any $\varphi \in \varPhi$, where $\rho_\frac{\pi}{2}$ is the rotation through a $\frac{\pi}{2}$ angle.
\end{itemize}
Let us set $\bar{\varphi} = |\sin{x}|$ and $\bar{\psi}= \sin^2{x}$.
As we can see in Figures 1 and 2, the functions $F_i(\bar{\varphi})$ and $F_i(\bar{\psi})$ have the same persistence diagrams for $i = 1,2$. In order to distinguish $\bar{\varphi}$ and $\bar{\psi}$, we define the operator $F:\varPhi \longrightarrow \varPhi$ by setting $F(\varphi):=M_1(F_1,F_2)(\varphi)= \frac{F_1(\varphi) + F_2(\varphi)}{2}$. In particular, 
\begin{equation}
    F(\bar{\varphi}):=M_1(F_1,F_2)(\bar{\varphi})= \frac{F_1(\bar{\varphi}) + F_2(\bar{\varphi})}{2}= \frac{|\sin{x}| + |\cos{x}|}{2}
\end{equation}
and
\begin{equation}
F(\bar{\psi}):=M_1(F_1,F_2)(\bar{\psi})= \frac{F_1(\bar{\psi}) + F_2(\bar{\psi})}{2}= \frac{sin^2{x} + \cos^2{x}}{2}=\frac{1}{2}.
\end{equation}

We can easily check that $F(\bar{\varphi})$ and $F(\bar{\psi})$ have different persistence diagrams; thus $F$ allows us to distinguish between $\bar{\varphi}$ and $\bar{\psi}$. All this proves that the use of the operator $M_1$ can increase the information, letting $F_1$ and $F_2$ cooperate.

\begin{figure}[htbp]
\centering
\includegraphics[scale=0.35]{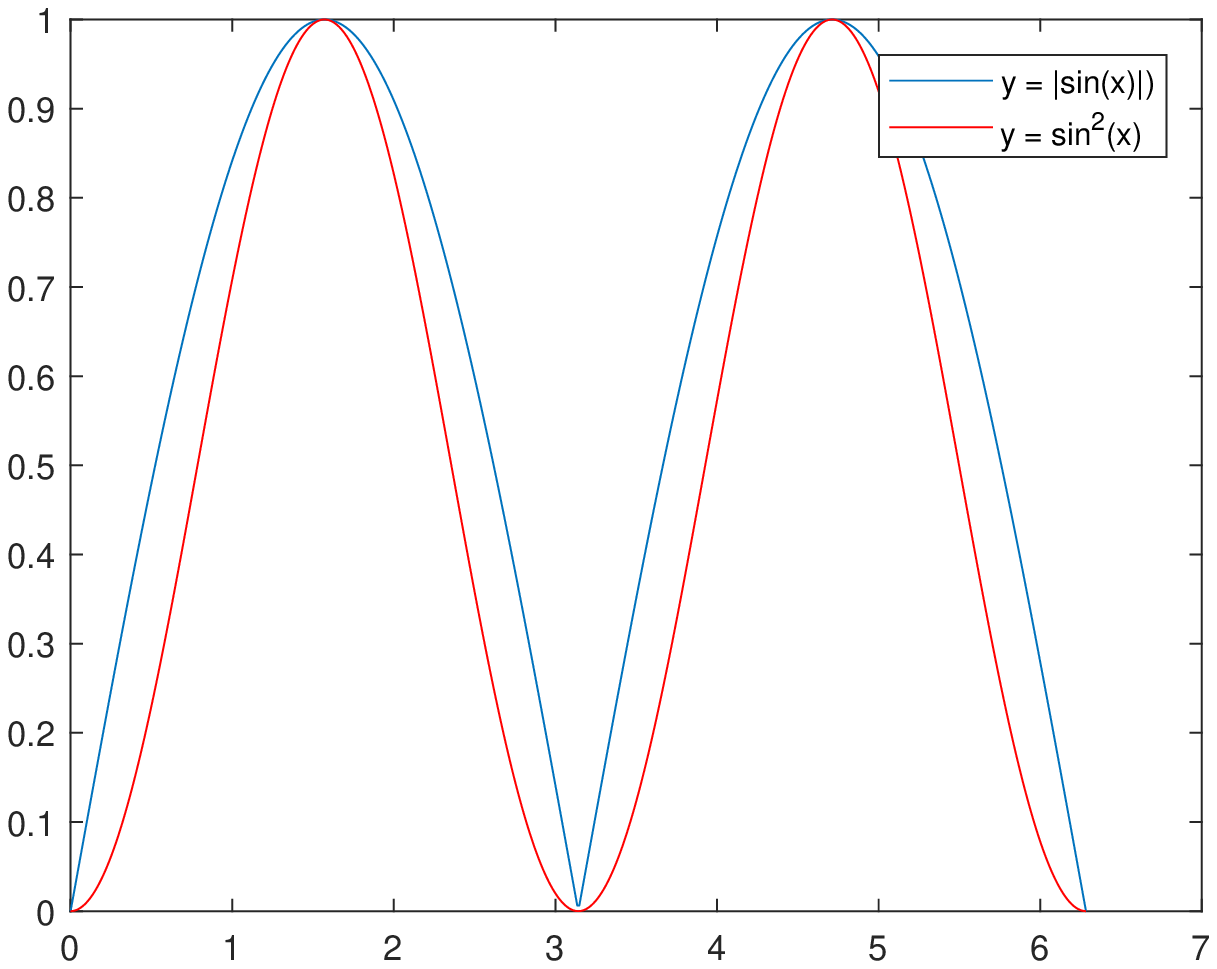}
\qquad\qquad
\includegraphics[scale=0.35]{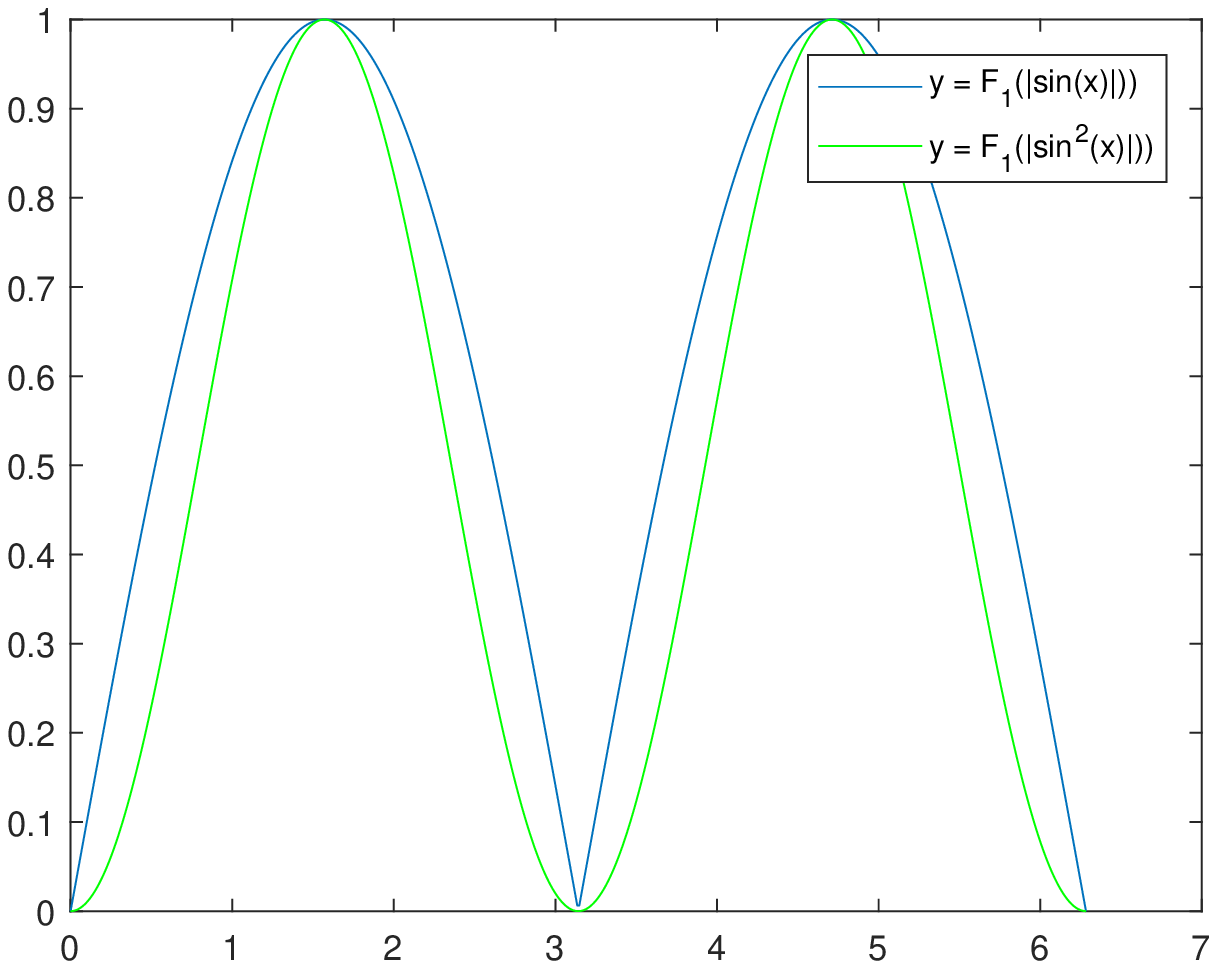}
\caption{On the left: $\bar{\varphi}$ and $\bar{\psi}$ have the same persistence diagrams. On the right: $F_1(\bar{\varphi})$ and $F_1(\bar{\psi})$ have the same persistence diagrams.}
\end{figure}

\begin{figure}[htbp]
\centering
\includegraphics[scale=0.35]{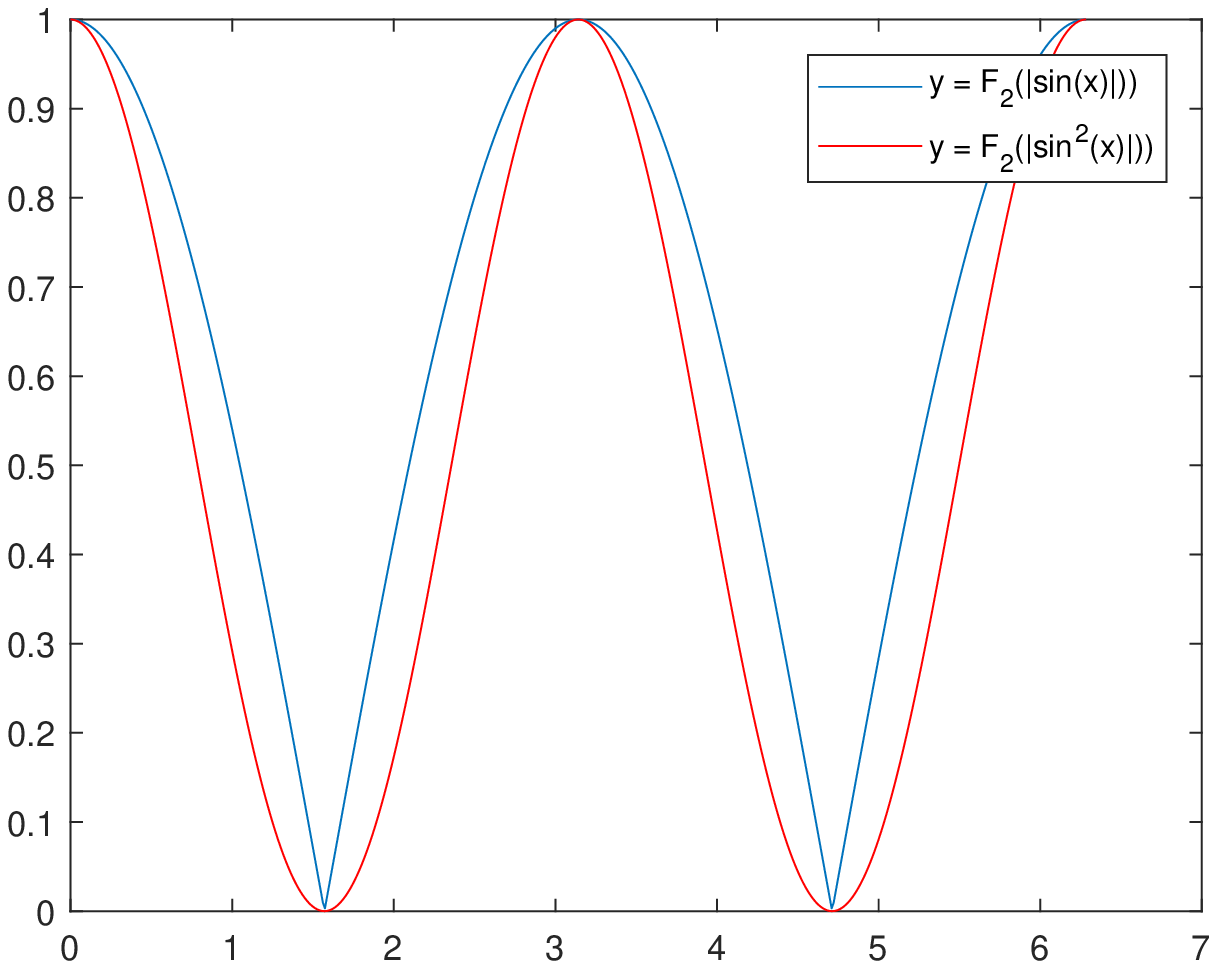}
\qquad\qquad
\includegraphics[scale=0.35]{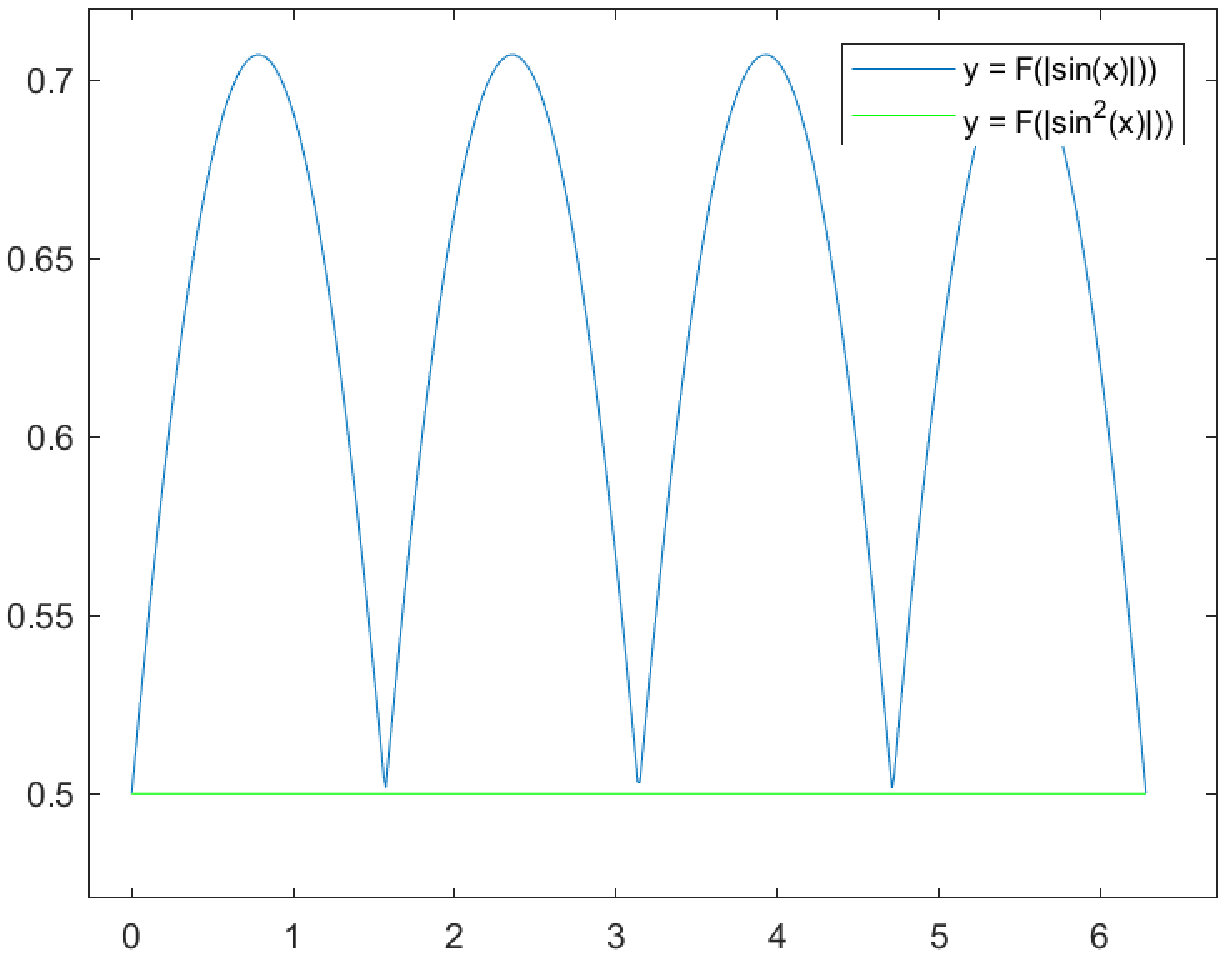}
\caption{On the left: $F_2(\bar{\varphi})$ and $F_2(\bar{\psi})$ have the same persistence diagrams. On the right: the persistence diagrams of $F(\bar{\varphi})$ and $F(\bar{\psi})$ are different from each other.}
\end{figure}

A similar argument still holds for values of $p$ greater than one. Under the same hypotheses about $\varPhi$, we can consider the same GENEOs $F_1$, $F_2$ and the functions $\bar{\varphi}= |\sin x|$ and $\hat{\psi}= (\sin^2 x)^\frac{1}{p}$. For the sake of simplicity, we fixed $p=3$ in order to represent the following figures. As we can see in Figures 3 and 4, we cannot distinguish $\bar{\varphi}$ and $\hat{\psi}$ by using persistent homology since their persistence diagrams coincide. Neither applying $F_1$ nor $F_2$ can help us, but when we apply $M_p(F_1,F_2)$ we can distinguish $\bar{\varphi}$ from $\hat{\psi}$ by means of their persistence diagrams (see Figure 4).

\begin{figure}[htbp]
\centering
\includegraphics[scale=0.35]{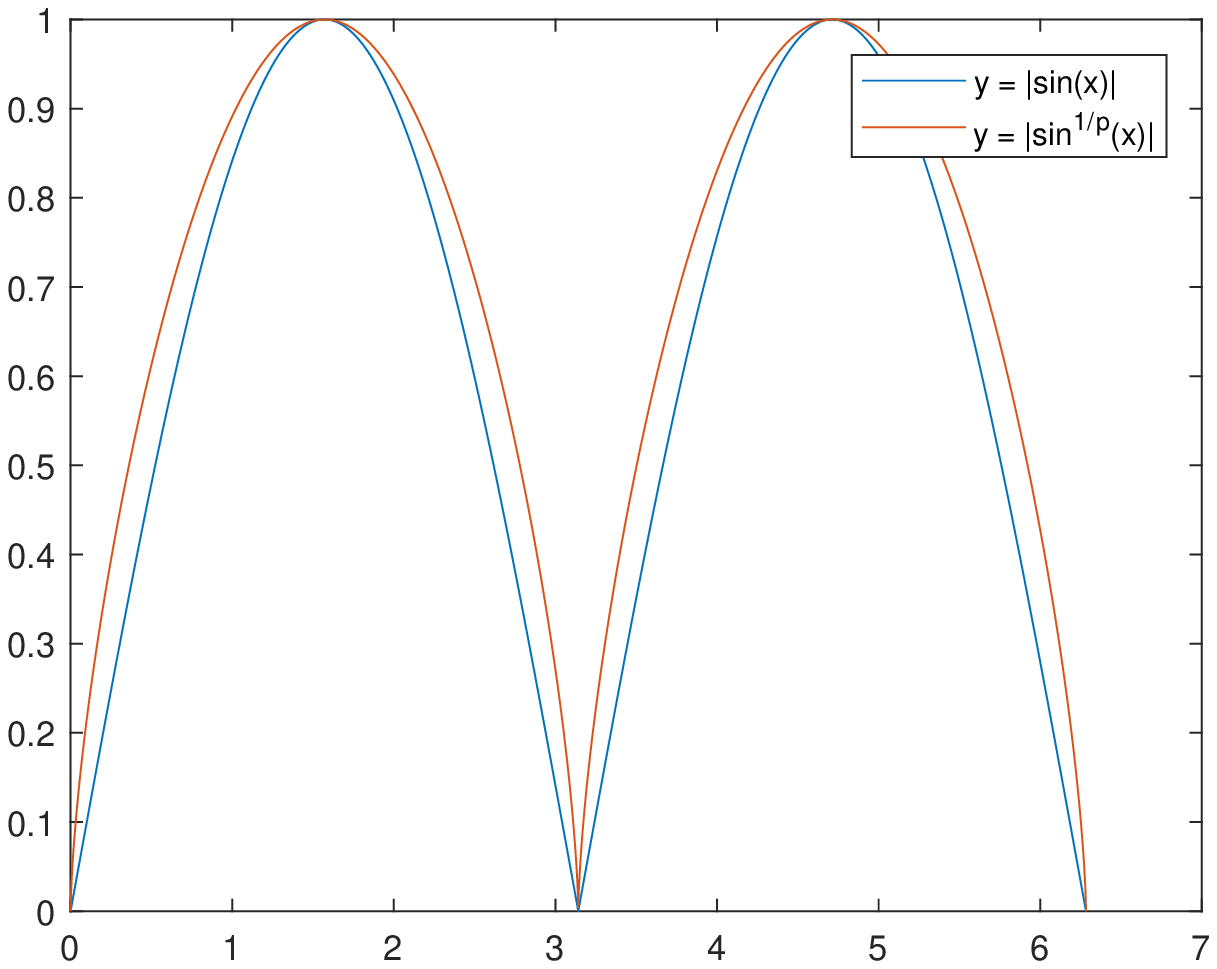}
\qquad\qquad
\includegraphics[scale=0.35]{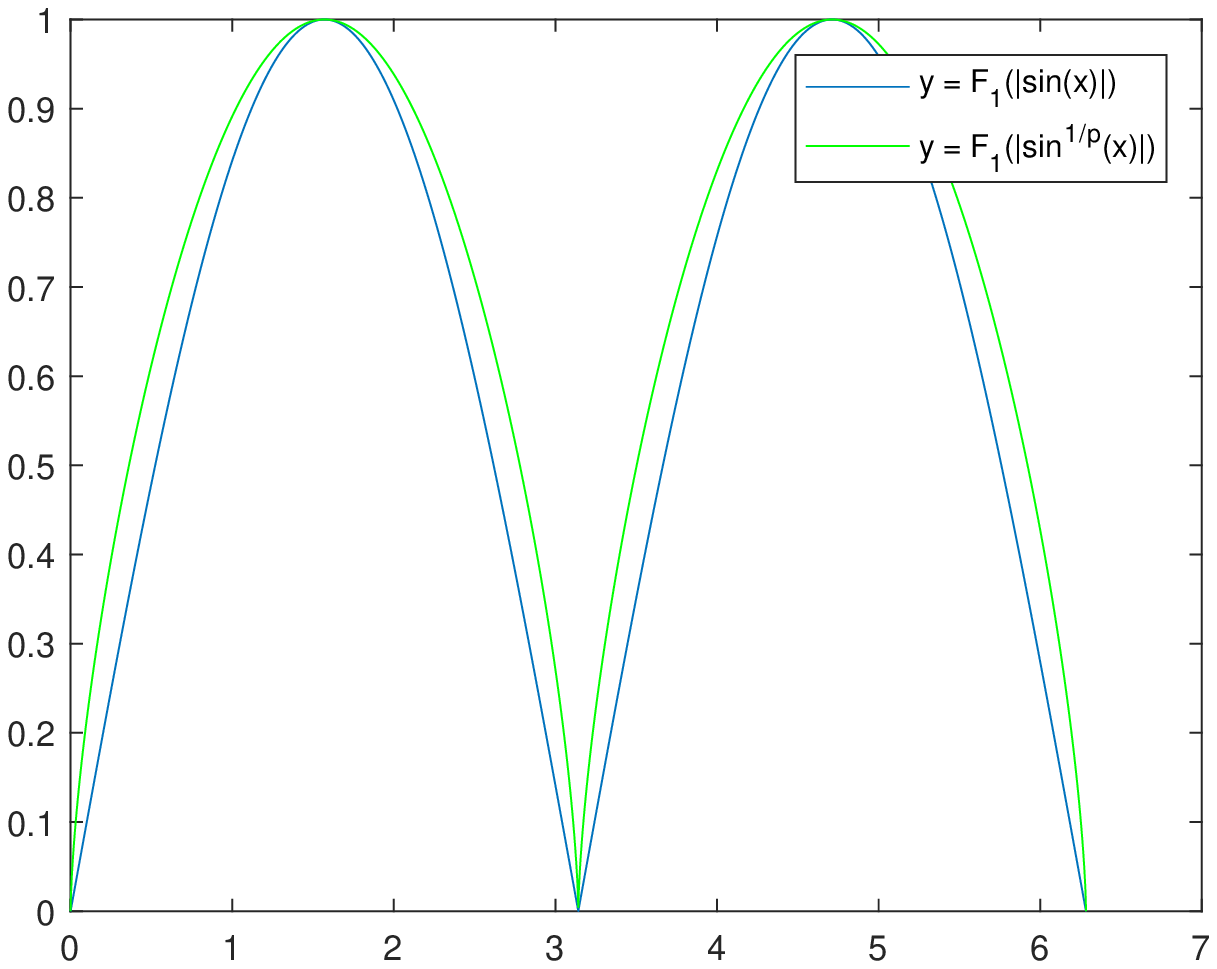}
\caption{On the left: $\bar{\varphi}$ and $\hat{\psi}$ have hence the same persistence diagrams. On the right: On the right: $F_1(\bar{\varphi})$ and $F_1(\hat{\psi})$ have the same persistence diagrams.}
\end{figure}

\begin{figure}[htbp]
\centering
\includegraphics[scale=0.35]{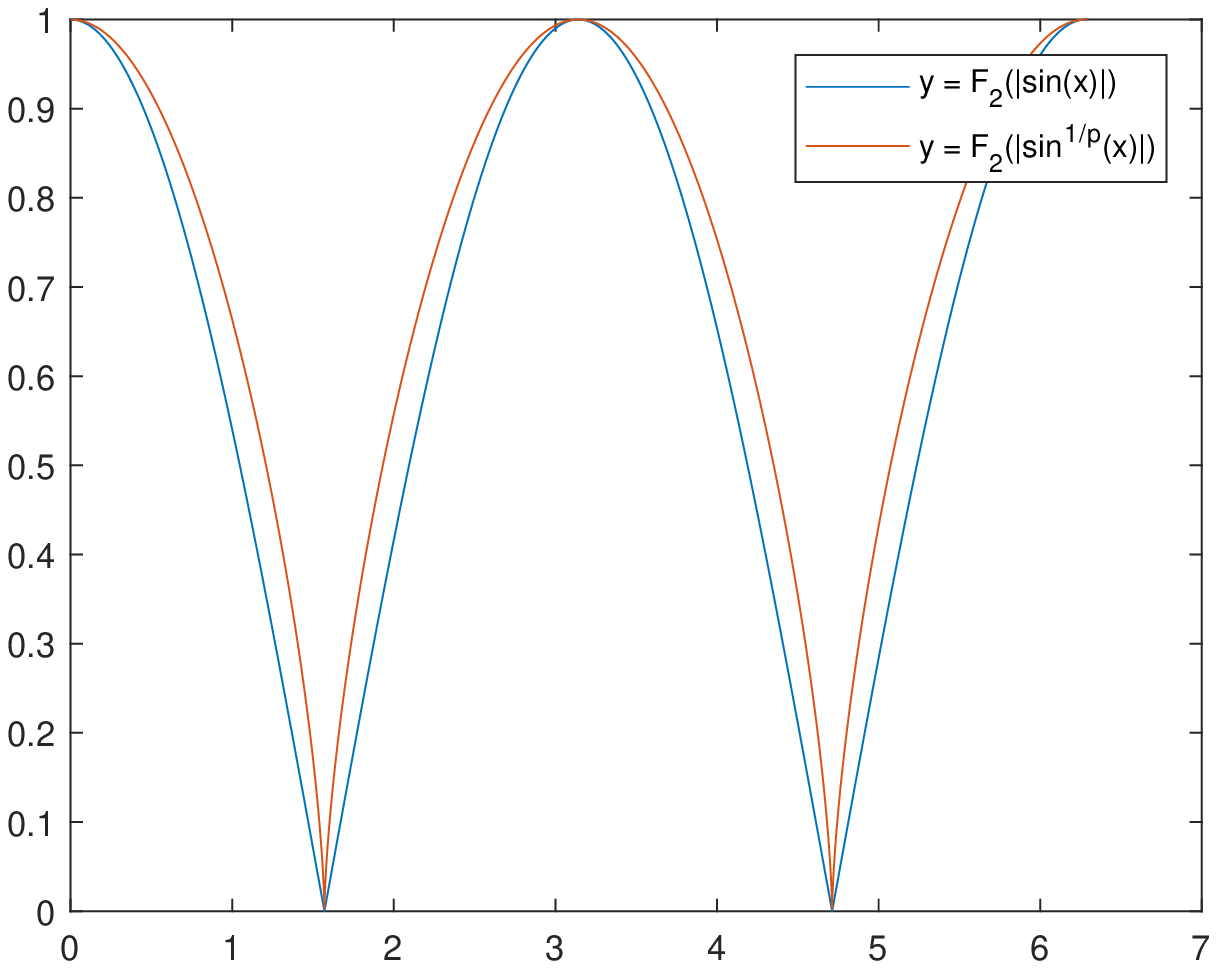}
\qquad\qquad
\includegraphics[scale=0.35]{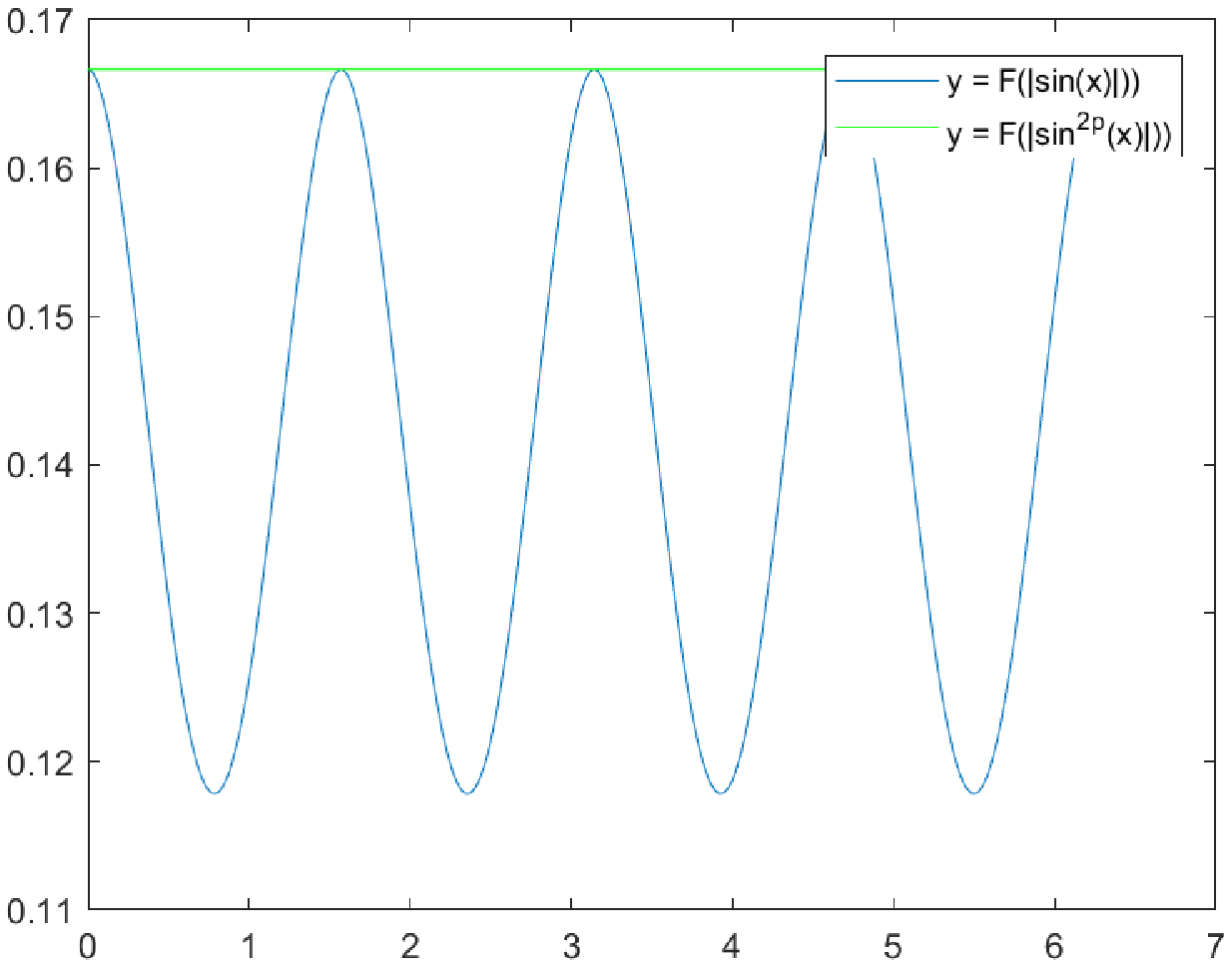}
\caption{On the left: $F_2(\bar{\varphi})$ and $F_2(\hat{\psi})$ have the same persistence diagrams. On the right: the persistence diagrams of $F(\bar{\varphi})$ and $F(\hat{\psi})$ are different from each other.}
\end{figure}

These examples justify the use of the previously defined power mean operators $M_p(F_1, \dots , F_n)$ to combine the information given by the operators $F_1, \dots , F_n$.

\section{Series of GENEOs}
\label{series}
First we recall some well-known results about series of functions.
\begin{theorem}
Let $(a_k)$ be a positive real sequence such that $(a_k)$ is decreasing and $\lim_{k \to \infty}a_k=0$. Let $(g_k)$ be a sequence of bounded functions from the topological space $X$ to $\mathbb{C}$. If there exists a real number $M>0$ such that
\begin{equation}
\left|\sum_{k=1}^n g_k(x)\right| \le M
\end{equation}
for every $x \in X$ and every $n \in \mathbb{N}$, then the series $\sum_{k=1}^{\infty} a_k g_k$ is uniformly convergent on $X$.
\end{theorem}

The second result ensures us that a uniformly convergent series of continuous functions is a continuous function.

\begin{theorem}
Let $(f_n)$ be a sequence of continuous function from a compact topological space $X$ to $\mathbb{R}$. If the series $\sum_{k=1}^{\infty} f_k$ is uniformly convergent, then $\sum_{k=1}^{\infty} f_k$ is continuous from $X$ to $\mathbb{R}$.
\end{theorem}

Now we can define a series of GENEOs. Let us consider a compact pseudo-metric space $(X,d)$, a space of real-valued continuous functions $\varPhi$ on $X$ and a subgroup $G$ of the group $\text{Homeo}(X)$ of all homeomorphisms from $X$ to $X$, such that if $\varphi \in \varPhi$ and $g \in G$, then $\varphi \circ g \in \varPhi$. Let $(a_k)$ be a positive real sequence such that $(a_k)$ is decreasing and $\sum_{k=1}^{\infty}a_k \le 1$.
Let us suppose that $(F_k)$ is a sequence of GENEOs for $(\varPhi,G)$ and that for any $\varphi \in \varPhi$ there exists $M(\varphi)>0$ such that 

\begin{equation}
    \left|\sum_{k=1}^n F_k(\varphi)(x)\right| \le M(\varphi)
\end{equation}
for every $x \in X$ and every $n \in \mathbb{N}$. These assumptions fulfill the hypotheses of the previous theorems and ensure that the following operator is well-defined.
Let us consider the operator $F: C_b^0(X,\mathbb{R}) \longrightarrow C_b^0(X,\mathbb{R})$ defined by setting
\begin{equation}
    F(\varphi):= \sum_{k=1}^{\infty} a_k F_k(\varphi).
\end{equation}
\begin{proposition}
If $F(\varPhi)\subseteq \varPhi$, then F is a GENEO for $(\varPhi,G)$.
\end{proposition}

\begin{proof}
\begin{itemize}
    \item Let $g\in G$. Since $F_k$ is $G$-equivariant for any $k$ and $g$ is uniformly continuous (because $X$ is compact), $F$ is $G$-equivariant:
    \begin{align*}
        F(\varphi \circ g) & =  \sum_{k=1}^{\infty} a_k F_k(\varphi\circ g)\\
                           & =  \sum_{k=1}^{\infty} a_k (F_k(\varphi)\circ g)\\
                           & =  \left( \sum_{k=1}^{\infty} a_k F_k(\varphi)\right) \circ g \\
                           & = F(\varphi) \circ g
    \end{align*}
    for any $\varphi \in \varPhi$.
    \item  Since $F_k$ is non-expansive for any $k$ and $\sum_{k=1}^\infty a_k \le 1$, $F$ is non-expansive:
    \begin{align*}
        \|F(\varphi_1) - F( \varphi_2)\|_\infty & = \left\| \sum_{k=1}^{\infty} a_k F_k(\varphi_1) -  \sum_{k=1}^{\infty} a_k F_k(\varphi_2) \right\|_\infty \\
        & = \left\| \lim_{n \to \infty} \left( \sum_{k=1}^n a_k F_k(\varphi_1) -  \sum_{k=1}^n a_k F_k(\varphi_2) \right) \right\|_\infty \\
        & = \lim_{n \to \infty} \left\| \sum_{k=1}^n a_k (F_k(\varphi_1) - F_k(\varphi_2)) \right\|_\infty \\
        & \le \lim_{n \to \infty} \sum_{k=1}^n (a_k \| F_k(\varphi_1) - F_k(\varphi_2)\|_\infty) \\
        & \le \lim_{n \to \infty} \sum_{k=1}^n (a_k \| \varphi_1 - \varphi_2 \|_\infty) \\
        & = \sum_{k=1}^\infty a_k \| \varphi_1 - \varphi_2 \|_\infty \\
        & \le \| \varphi_1 - \varphi_2 \|_\infty.
    \end{align*}
\end{itemize}

\end{proof}

\section*{Conclusions}
In this work we have illustrated some new methods to build new classes of $G$-equivariant non-expansive operators (GENEOs) from a given set of operators of this kind. The leading purpose of our work is to expand our knowledge about the topological space $\mathcal{F}(\varPhi,G)$ of all GENEOs. If we can well approximate the space $\mathcal{F}(\varPhi,G)$, we can obtain a good approximation of the natural pseudo-distance $d_G$ (Theorem \ref{maintheoremforG}). Searching new operators is a fundamental step in getting more information about the structure of $\mathcal{F}(\varPhi,G)$, and hence we are asked to find new methods to build GENEOs. Moreover, the approximation of $\mathcal{F}(\varPhi,G)$ can be seen as an approximation of the considered observer, represented as a collection of GENEOs. Many questions remain open. In particular, we should study an extended theoretical framework that involves GENEOs from the pair $(\varPhi,G)$ to a different pair $(\Psi, H)$. A future research about this is planned to be done.

\section*{Acknowledgment}
The research described in this article has been partially supported by GNSAGA-INdAM (Italy).

\bibliographystyle{model1-num-names}

\end{document}